\documentclass{article}
\usepackage{graphicx}
\usepackage{epstopdf}
\usepackage{amsmath}
\usepackage{amssymb}
\usepackage{bm}
\usepackage{bbm}
\usepackage{enumerate}
\usepackage[T1]{fontenc}
\usepackage[latin1]{inputenc}
\usepackage{hyperref}
\usepackage{color}
\usepackage{curves}
\usepackage{tikz}
\usepackage[center]{caption}
\usepackage{ntheorem}
\usepackage{float}
\usepackage{mathrsfs}
\usepackage{wrapfig}

\title{Overly determined agents prevent consensus in a generalized Deffuant model
	   on $\Z$ with dispersed opinions}
\author{Timo Hirscher\thanks{Research supported by a grant from the Swedish Research Council}
\\\normalsize Chalmers University of Technology}

\theoremstyle{break}
\newtheorem{theorem}{Theorem}[section]

\newtheorem{lemma}{Lemma}[section]

\theorembodyfont{\upshape}
\newtheorem{definition}{Definition}
\newtheorem*{remark}{Remark}
\newtheorem{example}{Example}[section]
\makeatletter
\let\c@proposition\c@theorem
\let\c@lemma\c@theorem
\let\c@example\c@theorem
\makeatother

\newenvironment{proof}{\noindent{\sc Proof:}}{\vspace{-0.5cm}~\hfill $\square$\vspace{0.5cm}}
\newenvironment{nproof}[1]{\noindent{\sc Proof #1:}}{\vspace{-1em}~\hfill $\square$\vspace{2em}}

\newcommand\N{\mathbb{N}}
\newcommand\R{\mathbb{R}}
\newcommand\Z{\mathbb{Z}}
\newcommand\E{\mathbb{E}\,}
\newcommand\Prob{\mathbb{P}}
\renewcommand\epsilon{\varepsilon}
\renewcommand\phi{\varphi}

\definecolor{darkblue}{rgb}{0,0,.5}
\hypersetup{colorlinks=true, breaklinks=true, linkcolor=blue, 
citecolor=darkblue, menucolor=blue, urlcolor=blue}

\begin{document}
\newpage
\maketitle
\begin{abstract}
During the last decades, quite a number of interacting particle systems have been introduced
and studied in the border area of mathematics and statistical physics. Some of these can be seen
as simplistic models for opinion formation processes in groups of interacting people. In the one
introduced by Deffuant et al.\ agents, that are neighbors on a given network graph, randomly meet
in pairs and approach a compromise if their current opinions do not differ by more than a given
threshold value $\theta$. We consider the two-sidedly infinite path $\Z$ as underlying graph and
extend former investigations to a setting in which opinions are given by probability distributions.
Similar to what has been shown for finite-dimensional opinions, we observe a dichotomy in the
long-term behavior of the model, but only if the initial narrow-mindedness of the agents is
restricted.
\end{abstract}



\section{Introduction}\label{intro}

The research field that became known as {\itshape opinion dynamics} originated from simple models
for interacting elementary particles established in statistical physics, introduced to figure out
how microscopic interaction rules lead to macroscopic properties of the whole system. Due to the
strong link between statistical mechanics and spatial stochastic processes, interest among
mathematicians was raised and in the course of a few decades an abundance of new models with
similar but qualitatively different interaction schemes was introduced and analyzed, primarily
by computer-based stochastic simulations. The survey article \cite{surv} gives a broad overview
of the different models and their analyses and applications.

Despite their radical limitations in terms of complexity, these models attracted more and more the
attention of the social sciences and were used to describe group behavior on an elementary level
and to explain real life phenomena. In 2000, Deffuant et al.\ \cite{Model} suggested a simple model
that features a bounded confidence restriction: Neighbors talk to each other in pairs and their
opinions are updated towards a compromise only if the opinions they hold just before they meet are
not further apart than a given threshold. This is meant to capture the realistic phenomenon
that people tend to modify their attitude on a specific topic when talking to others, but
not if they consider the views of their discussion partner as so alien as to seem like complete
nonsense.

\vspace*{1em}
In mathematical terms, the {\em Deffuant model} is structured as follows: First, we are
given a simple connected graph $G=(V,E)$, that shapes the underlying network. The
vertices are understood to represent agents holding individual opinions on a certain topic. The
edges of the graph are supposed to represent the connections between these individuals and entail
a possible mutual influence among neighbors. The vertex set $V$ can be either finite or countably
infinite. In the latter case the maximal degree in $G$ is commonly assumed to be bounded.

Then there are two model parameters: the already mentioned confidence bound $\theta>0$ and the
convergence parameter  $\mu\in(0,\tfrac 12]$, shaping the step size towards a compromise when two
opinions are updated. Opinions usually take values in $\R$. A higher-dimensional analog was
considered in \cite{multidim} and here we will extend the model further. For these generalizations,
we need to specify a metric $d$ that is used to measure the distance of two opinions and takes
over the task of the absolute value in the original model.

The first source of randomness is the configuration of initial opinions. Even though there have been
attempts to look at settings with dependent initial opinions (see for example Section 2.2 in \cite{Deffuant})
the usual setting is to take i.i.d.\ initial opinions which then evolve dependencies by interacting.
The opinion value at vertex $v\in V$ and time $t\geq0$ will be denoted by $\eta_t(v)$.

The second source of randomness in the model is the succession of pairwise encounters. On finite graphs,
the next pair of neighbors to meet is picked uniformly at random. On infinite graphs, the corresponding
equivalent is to assign i.i.d.\ Poisson processes on all edges of the graph: Whenever a {\em Poisson event}
occurs on the edge $e=\langle u,v \rangle$, i.e.\ a jump in the Poisson process associated with $e$,
the agents $u$ and $v$ interact in the following way: Assume the event happens at time $t$ and the
opinions of $u$ and $v$ just before are given by $\eta_{t-}(u)=\lim_{s\uparrow t}\eta_s(u)=:a$ and
$\eta_{t-}(v)=\lim_{s\uparrow t}\eta_s(v)=:b$ respectively. Then, depending on the distance of $a$ and
$b$, there might be an update according to the following rule:
\begin{equation*}
\eta_t(u) = \left\{ \begin{array}{ll}
a+\mu(b-a) & \mbox{if $d(a,b)\leq\theta$,} \\
a & \mbox{otherwise}
\end{array} \right.
\end{equation*}                     
and similarly \vspace{-0.75cm}\begin{align}\label{dynamics}\end{align}\vspace{-0.75cm}

\begin{equation*}
\eta_t(v) = \left\{ \begin{array}{ll}
b+\mu(a-b) & \mbox{if $d(a,b)\leq\theta$,} \\
b & \mbox{otherwise.}
\end{array} \right.
\end{equation*}
\vspace*{1em}

Given our assumptions, $E$ is countable, so there will almost surely be neither two simultaneous Poisson
events nor a limit point in time for the Poisson events on edges incident to one fixed vertex.
This guarantees the well-definedness of the process by \eqref{dynamics} for finite $G$. The
extension to infinite graphs with bounded degree is not immediately obvious but a standard argument,
see Thm.\ 3.9 in \cite{Liggett}.

When it comes to the long term behavior of the system, it is quite natural to ask whether the
agents will form a consensus to which all the opinions converge or not. Let us properly define and
distinguish the following two opposing asymptotics of the Deffuant model as time tends to infinity:

\begin{definition}\label{states}
\begin{enumerate}[(i)]
\item {\itshape Disagreement}\\
There will be finally blocked edges, i.e.\ edges $e=\langle u,v\rangle\in E$ s.t.
$$d(\eta_t(u),\eta_t(v))>\theta,$$
for all times $t$ large enough. Hence the vertices fall into different opinion groups,
that are incompatible with neighboring ones.
\item {\itshape Consensus}\\
The value at every vertex converges, as $t\to\infty$, to a common limit $l$, where
$$l=\begin{cases}\frac{1}{|V|}\sum_{v\in V}\limits\eta_0(v),&\text{if }G\text{ is finite}\\
                 \E\eta,&\text{if }G\text{ is infinite}\end{cases}$$
and $\mathcal{L}(\eta)$ denotes the distribution of the initial opinion values.\end{enumerate}
\end{definition}
Even though these two regimes intuitively seem to be complementary, for infinite graphs it is
far from obvious that the asymptotics of the model is necessarily given by one or the other
(cf.\ Def.\ 1.1 in \cite{Deffuant} and also the remark at the end of Section \ref{sec:td}).

In this paper, we are going to consider the two-sidedly infinite path $\Z$ as underlying graph,
i.e.\ $V=\Z$ and $E=\{\langle v,v+1\rangle, v\in\Z\}$. The first result for this setting was
published in 2011 and is due to Lanchier \cite{Lanchier}, who showed a sharp phase transition
from almost sure disagreement to almost sure consensus at $\theta=\frac12$, given initial opinions,
that are i.i.d.\ $\text{\upshape unif}([0,1])$. Shortly thereafter, Häggström \cite{ShareDrink}
reproved Lanchier's findings using a quite different approach; then Häggström and Hirscher
\cite{Deffuant} extended them to general univariate distributions for $\mathcal{L}(\eta)$.
In \cite{multidim}, the case of vector-valued opinions and different distance measures was
examined. 

One aspect that could be considered unrealistic in these models is the fact that even though
opinions are random, for a fixed realization they were given by numbers or vectors, hence
entirely determined -- not doing justice to the extremely common phenomenon of uncertainty in
people's opinions. In what follows, we are going to introduce and analyze a variant of the Deffuant
model on $\Z$, where the opinions are given by random absolutely continuous measures on $[0,1]$.
The support of these measure-valued opinions can be seen to represent uncertainty: the more
concentrated the measure, the more determined the agent.
\vspace*{1em}

As a general preparation for an extension of the model in this direction, in Section \ref{dist&conv}
we will introduce the total variation distance (which will be used as distance measure) and
recall a Strong law of large numbers (SLLN) for continuous densities, due to Rubin, replacing
the common SLLN which was a crucial ingredient in the case of finite-dimensional opinions.

The model with measure-valued opinions is outlined in Section \ref{measure_model}. We
consider random symmetric triangular distributions as initial opinions and find that for
this setting overly determined agents (i.e.\ agents whose initial opinion is concentrated on
sufficiently short intervals) prevent consensus for all $\theta\in[0,1)$, cf.\
Theorem \ref{unrestricted}.

The results for finite-dimensional opinion spaces listed above will be sketched in more detail
in Section \ref{background}. Central ideas from \cite{ShareDrink} will be presented as
they prove to be useful in our setting as well.

In Section \ref{sec:td} the main result for the setting with unrestricted symmetric triangular
distributions, Theorem \ref{unrestricted}, is proved: We show that the behavior of the model is
trivial in this case as extremely determined agents will have and keep a total variation distance
close to $1$ to their neighbors' opinions.

If we put a restriction on the initial determination of
the agents by disallowing triangular distributions that have a support of length less than
a fixed value $\gamma$, the familiar phenomenon of a phase transition in $\theta$ from a.s.\ 
disagreement to a.s.\ consensus reappears. This case, as well as the precise dependency of
the threshold value $\theta_\mathrm{c}$ on the parameter $\gamma$ are elaborated in Section
\ref{sec:rrtd}. In the final section, we breifly discuss possible other initial configurations
and earlier attempts to incorporate inhomogeneous open-mindedness of the agents.

\section{Distance and convergence of absolutely continuous random measures}\label{dist&conv}
As indicated, we want to generalize the Deffuant model on $\Z$ further by looking at opinions
that are no longer numbers or vectors but probability distributions instead. These random
distributions can be seen to shape indeterminacy in the agents: Even with initial opinion
profile and sequence of encounters fixed, the opinion of an individual at a given time is not
a fixed value but a probability measure.
Initially, the agents are independently assigned random measures from a common distribution.
When they meet and their current opinion measures do not differ by more than $\theta$, with respect
to a fixed metric on probability measures, the new opinions will be given by convex combinations
of the old ones, just as described in (\ref{dynamics}).

In order to quantify the difference between two distributions there are quite a few metrics to choose
from. The so-called total variation distance is among the most common ones.

\begin{definition}
	Let $\mu$ and $\nu$ be two probability distributions on a set $S$. The {\em total variation distance}
	between the two measures is then defined as
	$$\lVert\mu-\nu\rVert_{\mathrm{TV}}:=\sup_{A\subseteq S}|\mu(A)-\nu(A)|.$$
\end{definition}

As the total variation distance of two probability distributions is a number in $[0,1]$, the
non-trivial values for $\theta$ lie in $(0,1)$ for this model. Further, note that the total
variation distance of two probability distributions $\mu$ and $\nu$, that are absolutely
continuous with respect to the Lebesgue measure on $\R$ and have densities $f$ and $g$
respectively, is given by
$$\lVert\mu-\nu\rVert_{\mathrm{TV}}=\frac12 \int_{\R}\big|f(x)-g(x)\big|\;\mathrm{d}x.$$
In addition, if $\mu$ and $\nu$ are distributions on $[0,1]$, we can immediately conclude
$\lVert\mu-\nu\rVert_{\mathrm{TV}}\leq\tfrac12\, \lVert f-g \rVert_{\infty}$,
where $\lVert f\rVert_\infty= \sup_{x\in[0,1]} \big|f(x)\big|$ denotes the supremum norm on 
$[0,1]^{\R}$.

\vspace*{1em}
To be able to transfer the findings from the Deffuant model on $\Z$ featuring real- or vector-valued
opinions, we further need an equivalent for the Strong law of large numbers (SLLN) geared towards the
densities of random measures. The following result of Rubin \cite{SLLN} serves our purposes.

Let $U, U_1, U_2,\dots$ denote a sequence of independent, identically distributed random variables
with values in an arbitrary space $Y$. Given a compact topological space $X$, consider a map
$f: X\times Y\to\R$, that is measurable in the second argument for each $x\in X$.
\begin{theorem}[SLLN for continuous densities]\label{SLLN}
	If there exists an integrable function $g$ on $Y$ such that $|f_y(x)|<g(y)$ for all $x\in X$ and $y\in Y$, as
	well as a sequence of measurable sets $(S_i)_{i\in\N}$ with
	$$\Prob\Big(U\in \bigcap_{i\in\N}S_i^\mathrm{\;c}\Big)=0,$$
	and the property that $\{f_y(\,.\,),\; y\in S_i\}$ is equicontinuous on $X$ for all $i\in\N$,
	then with probability 1,
	$$\lim_{n\to\infty} \frac1n \sum_{i=1}^n f_{U_i}(x)= \E f_U(x)$$
	uniformly in $x\in X$ and the limit function is continuous.
\end{theorem}
The measure with density $\E f_U$ is commonly called {\em intensity} or {\em intensity measure},
see e.g.\ Section 1.2 in \cite{rdmmeasures} for a more general introduction.
\pagebreak

\section{Initial opinions given by random triangular distributions}\label{measure_model}

For concreteness, let us pick the initial opinions from a specific class of absolutely continious
distributions. A rather natural choice, departing from real-valued opinions (which can be seen as
Dirac delta measures), are symmetric triangular distributions on random subintervals of $[0,1]$, the
endpoints of which are chosen uniformly from $[0,1]$.

More precisely, let us consider the initial opinions $\{\eta_0(v),\;v\in\Z\}$ to be picked in the
following way:
Consider $\{U(v),\;v\in\Z\}$ to be an i.i.d.\ sequence of $\text{\upshape unif}([0,1]^2)$ random vectors.
The node $v$ will be assigned  an initial opinion given by the random absolutely continuous probability measure
with density
\begin{align}
f_0^{(v)}(x)&=\begin{cases}
0,&x\notin(m,M)\\
\big(\tfrac{2}{M-m}\big)^2\cdot(x-m),&x\in(m,\tfrac{m+M}{2}]\\
-\big(\tfrac{2}{M-m}\big)^2\cdot(x-M),&x\in(\tfrac{m+M}{2},M)\\
\end{cases}\\
&=\label{densdef}\tfrac{2}{|y-z|}\cdot\big(1-\tfrac{2}{|y-z|}\cdot\big|x-\tfrac{y+z}{2}\big|\big)^+,\quad x\in[0,1],
\end{align}
where $U(v)=(y,z)$ and $m:=\min\{y,z\},\ M:=\max\{y,z\}$, see Figure \ref{dens}.

\begin{figure}[H]
	\centering
	\includegraphics[scale=0.83]{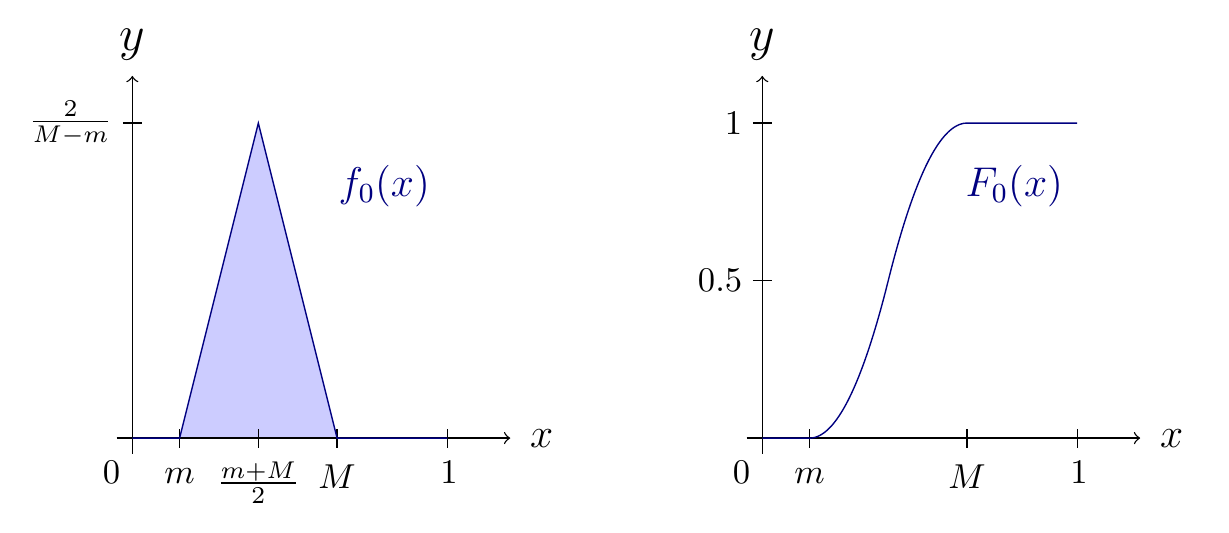}
	\caption{The density and distribution function of a symmetric triangular distribution on
		$[m,M]$.}\label{dens}
\end{figure}

Seen from a different angle, to get the initial opinion of a fixed agent, we first choose a central
opinion value $C$ uniformly from $[0,1]$ and then a spread for the support of the distribution uniformly
among $[0,\min\{C,1-C\}]$. That this procedure is equivalent to the one described above is an immediate
consequence of the change of variable formula, see the proof of Lemma \ref{density-la} (especially Figure
\ref{T}) for more details.

Note that this model features two qualitatively different forms of extreme initial opinions: One the one hand
-- as in the original model -- the agents can have opinions lying at the edges of the spectrum (i.e.\
concentrated close to 0 or 1 in this case), on the other an individual opinion can be very determined
in the sense that $U(v)$ is close to the diagonal (i.e.\ $|y-z|$ very small), which provokes a highly concentrated
density and necessarily a large distance to a vast majority of possible initial opinions.

This effect is quite realistic: Irrespectively of their opinion being exceptional or mainstream,
people that are extremely narrow-minded or determined are usually neither willing to consider the opinion
nor to accept the arguments of others, let alone to compromise. In this sense, even though the mathematics are
closely related to the case of finite-dimensional opinions, the extension of the model to measure-valued
opinions introduces an additional real life phenomenon.

For symmetric triangular distributions on $[0,1]$ without any restriction on the minimal length of
their support, we are going to show that the model exhibits a trivial behavior:

\begin{theorem}\label{unrestricted}
	Consider the Deffuant model on $\Z$, where the initial opinions are given by independently
	assigned random triangular distributions as described in \eqref{densdef}. Then for all $\theta\in[0,1)$,
	the system almost surely approaches disagreement in the long run if the total
	variation distance is used to measure the distance between two opinions.
\end{theorem}

For the proof of this result, we refer the reader to Section \ref{sec:td}. Since this setting allows to
reuse results from the finite-dimensional case, we first want to give a brief overview of these in the following
section.

\section{Background}\label{background}

As mentioned in the introduction, the first analytic result about consensus formation in the Deffuant model
on $\Z$ was established by Lanchier \cite{Lanchier} and deals with opinion profiles that are initially given
by an i.i.d.\ sequence of $\mathrm{unif}([0,1])$ random variables. The distance between two opinions was
taken to be the absolute value of their difference. Häggström \cite{ShareDrink} used different
techniques to reprove and slightly sharpen this result.
His arguments were later adapted to accommodate other univariate
initial distributions as well, leading to an analog covering all marginal distributions that
have a first moment $\E\eta_0\in\R\cup\{-\infty,+\infty\}$, see Thm.\ 2.2 in \cite{Deffuant}:

\begin{theorem}\label{gen}
	Consider the Deffuant model on the graph $(\Z,E)$, where $E=\{\langle v,v+1\rangle, v\in\Z\}$,
	with fixed parameter $\mu\in(0,\tfrac12]$. Let the initial configuration be given by an i.i.d.\ sequence
	of real-valued random variables, having the common distribution $\mathcal{L}(\eta_0)$, and the distance of two
	opinions by the absolute value of their difference.
\begin{enumerate}[(i)]
	\item Given a bounded distribution $\mathcal{L}(\eta_0)$ with expected value $\E\eta_0$, let $[a,b]$ denote
	the smallest closed interval containing its support. If $\E\eta_0$ does not lie in the support, let 
	$I\subset[a,b]$ denote the maximal, open interval with $\E\eta_0\in I$ and $\Prob(\eta_0\in I)=0$.
	In this case, set $h$ to be the length of $I$, otherwise set $h=0$.
	
	Then the critical value for $\theta$, marking the phase transition from a.s.\ disagreement to a.s.\ 
	consensus, becomes $\theta_{\mathrm{c}}=\max\{\E\eta_0-a,b-\E\eta_0,h\}$.
	The common limit value in the supercritical regime is $\E\eta_0$.
	\item Suppose the distribution $\mathcal{L}(\eta_0)$ is unbounded but its expected value exists, i.e.\
	$\E\eta_0\in\R\cup\{-\infty,+\infty\}$.
	Then the Deffuant model with arbitrary fixed parameter $\theta\in(0,\infty)$ will a.s.\ behave
	subcritically, meaning that disagreement will be approached in the long run.
\end{enumerate}
\end{theorem}

With an appropriate adaptation to the more involved geometry of vector-valued opinions, the main
ideas in \cite{ShareDrink} further served to establish similar results for the Deffuant model
on $\Z$ with opinion space $\R^d,\ d>2$, and more general distance measures, see 
Thm.\ 3.15 and Thm.\ 4.11 in \cite{multidim}.

 \vspace*{1em}\noindent
  Since the same line of reasoning was used in both \cite{ShareDrink}, \cite{Deffuant} and \cite{multidim}
  to derive the results for finite-dimensional opinion spaces we just mentioned, let us now take a closer
  look on the used key concepts and crucial auxiliary results. They form the base for most of the
  conclusions we will be able to draw in the case of infinite-dimensional opinions.

  In \cite{ShareDrink}, Häggström presents two central ideas, whose effect turns out to be highly limited
  to paths, but combined they prove to be quite powerful in the analysis of the Deffuant model on
  the infinite path $\Z$. The first one is the notion of {\em flat points}:
  
  \begin{definition}
   Consider the initial i.i.d.\ configuration $\{\eta_0(u)\}_{u\in\Z}$ with common marginal
   distribution $\mathcal{L}(\eta_0)$ on an opinion space $X$ (e.g.\ $\R$ or $\R^d$), which we consider to
   be equipped with the metric $\rho$.
   Under the premise that the mean $\E\eta_0$ of the initial distribution exists and given $\epsilon>0$,
   a vertex $v\in\Z$ is called {\em$\epsilon$-flat to the right} (with respect to the initial configuration),
   if for all $n\geq0$:
   \begin{equation}\label{rflat}
     \frac{1}{n+1}\sum_{u=v}^{v+n}\eta_0(u)\in B_\epsilon\big(\E\eta_0\big),
   \end{equation}
   where $B_r(x):=\{y\in X,\;\rho(x,y)\leq\epsilon\}$ denotes the (closed) $\rho$-ball around $x\in X$
   with radius $r>0$.
   A vertex $v$ is called {\em$\epsilon$-flat to the left} if the above condition is met with the sum
   running from $v-n$ to $v$ instead. Finally, $v$ is called {\em two-sidedly $\epsilon$-flat} if for all $m,n\geq0$
   \begin{equation}\label{tflat}
     \frac{1}{m+n+1}\sum_{u=v-m}^{v+n}\eta_0(u)\in B_\epsilon\big(\E\eta_0\big).
   \end{equation}
  \end{definition}\vspace*{1em}
   The crucial role vertices, that are one- or two-sidedly $\epsilon$-flat with respect to the initial
   configuration, can play in the further evolution of the configuration becomes more obvious in the
   light of the second key idea, the non-random pairwise averaging procedure Häggström \cite{ShareDrink}
   proposed to call {\em Sharing a drink} (SAD) on $\Z$.
   
   Glasses are put along the infinite path at all integers; the one at site $0$ is full, all others are empty. 
   Similarly to the Deffuant model, neighbors interact and share, but now we skip randomness and confidence
   bound: The procedure starts with the initial profile $\{\xi_0(v)\}_{v\in\Z}$,
   given by $\xi_0(0)=1$ and $\xi_0(v)=0$ for all $v\neq0$. In each step, we choose an edge, along which an
   update of the form \eqref{dynamics} is executed; more precisely, if we are given the profile
   $\{\xi_n(v)\}_{v\in\Z}$ after step $n$ and choose $\langle u,u+1\rangle$ for the next round, we get
   \begin{equation}\label{transf}\begin{array}{rl}\xi_{n+1}(u)&\!\!\!=\,(1-\mu)\,\xi_{n}(u)+\mu\,\xi_{n}(u+1),\\
                 \xi_{n+1}(u+1)&\!\!\!=\,\mu\,\xi_{n}(u)+(1-\mu)\,\xi_{n}(u+1),\\
                 \xi_{n+1}(v)&\!\!\!=\,\xi_{n}(v)\qquad\text{for all }v\notin\{u,u+1\}.\end{array}
   \end{equation}
   The resulting profiles, after we have performed this procedure a finite number of rounds, will be called
   {\em SAD-profiles}. Besides the facts that they feature only finitely many non-zero elements, the elements
   are all positive and sum to 1, there are less obvious properties that these profiles share which we will
   collect in the following lemma (for proofs, see Lemmas 2.2, 2.1 and Thm.\ 2.3 in \cite{ShareDrink}):
   \begin{lemma}\label{collection}
   	Consider the SAD-procedure on the infinite path $\Z$, started in vertex $v$, i.e.\ with
   	$\xi_0(u)=\delta_v(u),\ u\in V$. Then we get the following:
   	\begin{enumerate}[(i)]
   		\item All achievable SAD-profiles are unimodal.
   		\item If the vertex $v$ only shares the water to one side, it will remain a mode of the SAD-profile.
   		\item The supremum over all achievable SAD-profiles started with $\delta_v$ at another vertex $w$ equals $\tfrac{1}{d+1}$,
   		where $d$ is the graph distance between $v$ and $w$.
   	\end{enumerate}
   \end{lemma}
     
   The connection to the Deffuant model is established in Lemma 3.1 in \cite{ShareDrink}: The opinion value
   $\eta_t(0)$ at any given time $t>0$ can be written as a weighted average of the initial opinions, where the
   weights are given by the (random) SAD-profile which is dual to the dynamics in the Deffuant model in the sense
   that the order of updates has to be reversed.

   Combining this link with the concept of $\epsilon$-flatness makes it possible to derive the following
   crucial auxiliary results (which are obvious generalizations of intermediate results, established in the
   proofs of Prop.\ 5.1, as well as of Lemma 6.3 in \cite{ShareDrink}):
   
 \begin{lemma}\label{flat}
 		Consider the Deffuant model on $\Z$ with initial configuration be given by an i.i.d.\ sequence
 		of random variables having the common distribution $\mathcal{L}(\eta_0)$.
 	\begin{enumerate}[(i)]
 		\item If vertex $v$ is $\epsilon$-flat to the right with respect to the	initial configuration
 		      and does not interact with vertex $v-1$, its opinion stays inside $B_\epsilon\big(\E\eta_0\big)$.
 		      The same holds for $\epsilon$-flatness to the left and $v+1$ in place of $v-1$.
 		\item If vertex $v$ is two-sidedly $\epsilon$-flat with respect to the initial configuration,
 		      its opinion value will stay inside $B_{6\epsilon}\big(\E\eta_0\big)$, irrespectively of
 		      the dynamics.
 	\end{enumerate}
  \end{lemma}
	Without much further work, these findings can be used to analyze the behavior of the model featuring
	unrestricted symmetric triangular distributions, as we will see in the following section.

\section{Overly determined agents prevent consensus}\label{sec:td}

As the expectation of the initial distribution played a central role in the model featuring real- or
vector-valued opinions, we first have to get our hands on its counterpart in the context of random measures,
the intensity, before we can set about proving Theorem \ref{unrestricted}.

\begin{lemma}\label{density-la}
 Consider the absolutely continuous random measure $\eta$ to be given by the density
 \begin{equation}\label{randdens}
  f_U(x)=\tfrac{2}{|y-z|}\cdot\big(1-\tfrac{2}{|y-z|}\cdot\big|x-\tfrac{y+z}{2}\big|\big)^+,\quad x\in[0,1],
 \end{equation}
 where $U=(y,z)$ is taken uniformly from the unit square $[0,1]^2$, as introduced in \eqref{densdef}.
 Then its intensity measure (commonly denoted $\E\eta$) is given by the density
 \begin{equation*}
 \phi(x)=\begin{cases}-8\,\big[(1-x)\,\ln(1-x)+x\,\big(1-\ln(2)\big)\big],& x\in[0,\tfrac12]\\
                      -8\,\big[x\,\ln(x)+(1-x)\,\big(1-\ln(2)\big)\big],& x\in[\tfrac12,1]
 \end{cases}.
 \end{equation*}
\end{lemma}
\vspace*{-1em}
\begin{figure}[H]
	\hspace{3.5cm}
	\includegraphics[scale=0.8]{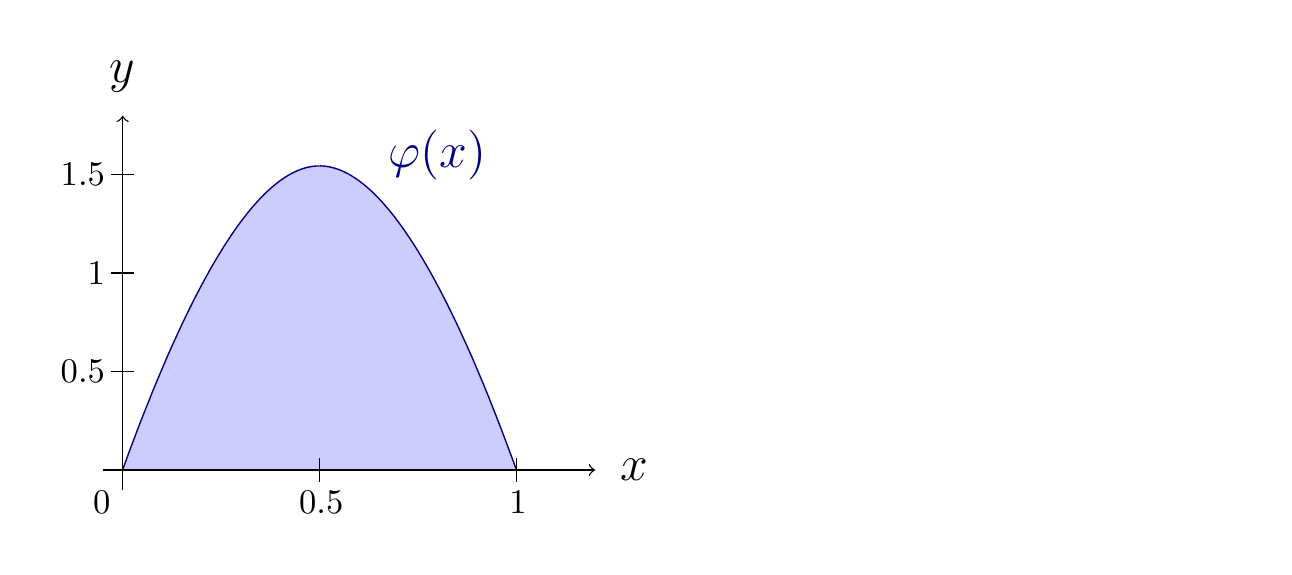}
	\caption{Density of the intensity measure corresponding to random symmetric triangular distributions
		on $[0,1]$.}
\end{figure}

\begin{proof}
	Fix $x\in(0,1)$. First of all, by symmetry, we can take $U$ to be uniform on the set
	$A:=\{(y,z)\in\R^2,\ 0\leq z\leq y\leq 1\}$. To further simplify the calculations, let us consider
	the simple linear transform $T((y,z))=\frac12\,(y+z,y-z),$ depicted in Figure \ref{T} below.
	\begin{figure}[h]
		\centering
		\includegraphics[scale=0.9]{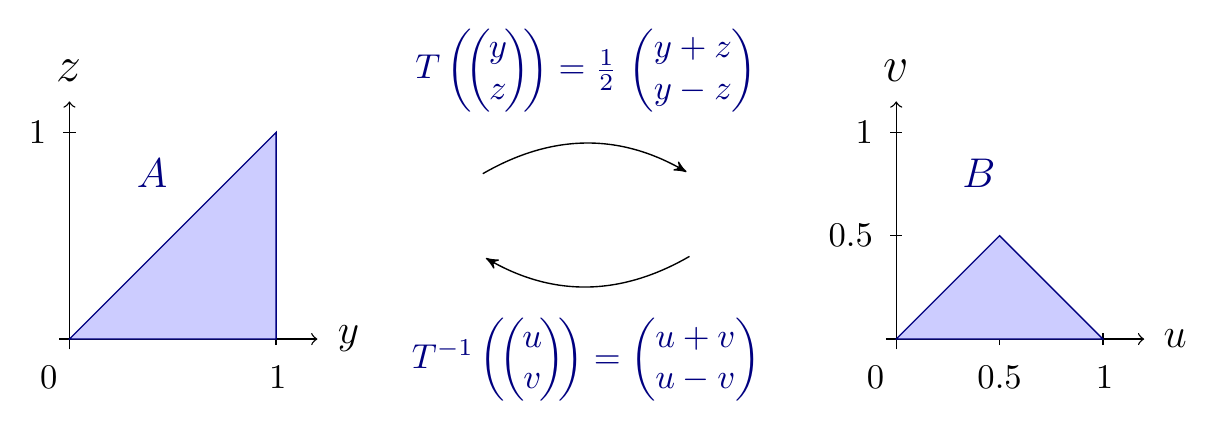}
		\caption{Using transform $T$ to consider the arithmetic mean and half the distance of two
			independent	$\mathrm{unif}([0,1])$ random variables instead.}\label{T}
	\end{figure}
	From the change of variable formula we know that $T(U)$ is uniform on the set
	$B:=\{(u,v),\ v\in[0,\tfrac12], u\in[v,1-v]\}$.
	
	Given the random density $f_U$ as in (\ref{randdens}), we can write
	$$f_{T(U)}(x)=\tfrac{1}{v}\cdot\big(1-\tfrac{1}{v}\cdot\big|x-u\big|\big)^+,\quad x\in[0,1]$$
	and conclude that $f_{(u,v)}(x)$ is non-zero for $(u,v)$ in $B_x=B_1\cup B_2\cup B_3$, where
	\begin{align*}
	B_1:=&\{(u,v),\ v\in[0,\tfrac{x}{2}],u\in[x-v,x+v]\}\\
	B_2:=&\{(u,v),\ v\in[\tfrac{x}{2},\tfrac{1-x}{2}],u\in[v,x+v]\}\\
	B_3:=&\{(u,v),\ v\in[\tfrac{1-x}{2},\tfrac{1}{2}],u\in[v,1-v]\}.
	\end{align*}
	Hence, for $x\in[0,\tfrac12]$, tedious but elementary calculations lead to
	\begin{align*}
	\phi(x):=&\;\E[f_U(x)]=4\cdot\iint\limits_{B_x} \frac1v-\frac{|x-u|}{v^2}\;\mathrm{d}u\,\mathrm{d}v\\
	&\hspace*{-0.37cm}=-8\,(1-x)\cdot\ln(1-x)-8x\,\big(1-\ln(2)\big)
	\end{align*}
	By symmetry around $x=\tfrac12$, the claim follows.
\end{proof}\vspace*{1.5em}

	\begin{nproof}{of Theorem \ref{unrestricted}}
		As usual, let $\eta_t(v)$ denote the opinion of individual $v\in\Z$ at time $t>0$ and
		further let $f_t^{(v)}$ be the density corresponding to this random measure.
		For any fixed $\delta>0$, let us define the random variables
		\begin{equation*}
		F_t^{(v)}(\delta):=\eta_t(v)([0,\delta])=\int_0^\delta f_t^{(v)}(x)\;\mathrm{d}x,\quad \text{for all } t>0,\ v\in \Z.
		\end{equation*}
		Their values lie in the interval $[0,1]$, which actually coincides with the support of their
		distributions. Furthermore, we know that $\{F_0^{(v)}(\delta),\;v\in\Z\}$ are i.i.d.\ random variables
		and Fubini's theorem gives
		\begin{equation}\label{fub}
		\E\big[F_0^{(v)}(\delta)\big]=\int_0^\delta \phi(x)\;\mathrm{d}x.
		\end{equation}
		We can disregard the case $\theta=0$, since there won't be any dynamics and hence a.s.\ disagreement.
		Given $\theta\in(0,1)$, define $\epsilon:=\tfrac12\,(1-\theta)>0$ and choose $\delta>0$ such that
		$\int_0^\delta \phi(x)\;\mathrm{d}x<\epsilon$.
		
		As mentioned above, the support of the distribution of $F_0^{(v)}(\delta)$
		is $[0,1]$ (without gaps), hence we can conclude as in Lemma 4.2 in \cite{ShareDrink} that any vertex is 
		(one-sidedly) $\epsilon$-flat with positive probability, with respect to the sequence 
		$\{F_0^{(v)}(\delta),\;v\in\Z\}$.
		Due to $\Prob(F_0^{(v)}(\delta)=1)=\delta^2>0$ and independence, the coincidence of the following two events
		occurs with positive probability for any $v\in\Z$:
		\begin{enumerate}[(a)]
			\item Vertex $v-1$ is $\epsilon$-flat to the left and vertex $v+1$ $\epsilon$-flat to the right w.r.t.\ 
			$\{F_0^{(v)}(\delta),\;v\in\Z\}$.
			\item $F_0^{(v)}(\delta)=1$
		\end{enumerate}
		Using part (i) of Lemma \ref{flat} and the same line of reasoning as in the proof of Prop.\ 5.1 in
		\cite{ShareDrink}, we can conclude that the edge $\langle v-1,v\rangle$ -- and similarly $\langle v,v+1\rangle$
		-- will be blocked forever, as
		\begin{align*}
		\lVert\eta_t(v)-\eta_t(v-1)\rVert_{\mathrm{TV}}&\geq \big|F_t^{(v)}(\delta)-F_t^{(v-1)}(\delta)\big|\\
		&=F_0^{(v)}(\delta)-F_t^{(v-1)}(\delta)\\
		&\geq 1-(\E\,F_0^{(v-1)}(\delta)+\epsilon)\\
		&>1-2\epsilon=\theta.
		\end{align*}
		From the fact that approaching disagreement is shift-invariant, hence a 0-1-event, we can conclude
		that for $\theta\in[0,1)$ there will a.s.\ be disagreement.
	\end{nproof}

\begin{remark}
 The trivial case $\theta=1$ in Theorem \ref{unrestricted} will surely not lead to blocked edges, so
 disagreement can be ruled out. However, this does not necessarily imply a consensus formation.
 The standard energy argument (as we will also use it in Lemma \ref{energyLa}) fails, since the
 random symmetric triangular distribution does not have a finite second moment, i.e.\ 
 $\E (f_U(x))^2=\infty$ for all $x\in(0,1)$.
 
 Using the results for univariate opinions once more, we can however conclude consensus for
 $\theta=1$ if we change to a different distance measure: the so-called Lévy-distance.
 Consider two probability distributions $\mu$ and $\nu$ on $[0,1]$. Their {\em Lévy-distance}
 $\rho(\mu,\nu)$ is defined as the infimum of the set
 $$\{\epsilon>0 \text{ s.t.\ } \mu([0,x-\epsilon])-\epsilon\leq\nu([0,x])\leq
 \mu([0,x+\epsilon])+\epsilon \text{ for all }x\in[0,1]\}.$$
  		
 To settle the case with $\theta=1$ and $\rho$ as distance measure, let us consider the univariate
 case, where $\{F_t^{(v)}(\delta)=\eta_t(v)([0,\delta]),\;t>0,\ v\in\Z\}$ are the opinions assigned to
 the agents: As there is no bounded confidence restriction, any encounter leads to an update and the
 update rule \eqref{dynamics} applies to both $\{\eta_t(v)\}_{v\in\Z}$ and 
 $\{F_t^{(v)}(\delta)\}_{v\in\Z}$.
 
 Hence, for any fixed $\delta\in[0,1]$, from Theorem \ref{gen} and \eqref{fub} we know that
 $F_t^{(v)}(\delta)$ converges to $\Phi(\delta):=\int_0^\delta \phi(x)\;\mathrm{d}x$ almost surely.
 Consequently, with probability $1$, it holds
 $$\lim_{t\to\infty}F_t^{(v)}(\delta)=\Phi(\delta)\quad\text{for all }v\in\Z,\ \delta \in [0,1]\cap\mathbb{Q}.$$
 Since all $F_t^{(v)}$ and $\Phi$ are continuous and increasing, this implies almost sure
 pointwise (in fact even uniform) convergence. In other words, for any $v\in\Z$ the opinion measure
 $\eta_t(v)$ converges with probability $1$ {\em vaguely} to the intensity measure $\E\eta$, having
 density $\phi$.
 As vague convergence of measures on a compact interval is metrized by the corresponding Lévy-metric
 (cf.\ for example Lemma 2 in \cite{grand}), this implies $\lim_{t\to\infty}\rho(\eta_t(v),\E\eta)=0$
 almost surely for all $v\in\Z$.\vspace*{1em}
  
 Note that a.s.\ consensus for $\theta=1$ and the Lévy-metric does not immediately imply a result for the
 total variation case, as $ \rho(\mu,\nu)\leq\lVert\mu-\nu\rVert_{\mathrm{TV}}$ for two probability
 measures $\mu$ and $\nu$.
\end{remark}

\section{Agents with bounded determination}\label{sec:rrtd}

In order to get a non-trivial phase transition in the parameter $\theta$, let us now consider a situation in
which all the agents feature at least a certain minimum of open-mindedness. This will be incorporated in our model
by disallowing the initial random measure to be concentrated on a subinterval of length less than $\gamma$, for a
fixed constant $\gamma\in(0,1)$. We will refer to these as random {\em restricted triangular distributions}.

Before we can show the main result, Theorem \ref{rrtd}, which states that there is a phase transition and the
precise threshold value for the parameter $\theta$, we need to study the altered intensity measure and verify
a few auxiliary results, needed to guarantee the existence of $\epsilon$-flat vertices (cf.\ Lemma \ref{flatLa}).

\begin{lemma}\label{phi_gamma}
	For fixed $\gamma\in(0,1)$, consider the absolutely continuous random measure $\eta_\gamma$ to be
	given by the density
	\begin{equation}\label{randdens_gamma}
	f_U(x)=\tfrac{2}{|y-z|}\cdot\big(1-\tfrac{2}{|y-z|}\cdot\big|x-\tfrac{y+z}{2}\big|\big)^+,\quad x\in[0,1],
	\end{equation}
	where $U=(y,z)$ is taken uniformly from the set $\{y,z\in[0,1],\;|y-z|\geq\gamma\}$ and note that this
	corresponds to the expression in \eqref{densdef}, conditional on the support being an interval of length
	at least $\gamma$.
	Then the density of its intensity measure $\E\eta_\gamma$ is given by the following expressions
	(assuming $0\leq x\leq\tfrac12$):
	\begin{enumerate}[1)]
	\item for $x\geq\gamma$
	\begin{equation*}\hspace*{-25mm}
	\phi_\gamma(x)=-\frac{8}{(1-\gamma)^2}\,\Big[(1-x)\,\ln(1-x)+x\,(1-\ln(2))+\frac{\gamma}{4}\Big],\\
	\end{equation*}
	\item for $x\geq 1-\gamma$
	\begin{equation*}\hspace*{-6.8mm}
	\phi_\gamma(x)=\begin{cases}
		-\frac{8}{(1-\gamma)^2}\,\Big[(1-x)\,\ln\gamma+x+\frac{1-2x}{2\gamma}-\frac{\gamma}{2}\Big],& x\leq\tfrac{\gamma}{2}\\[1.5mm]
		-\frac{8}{(1-\gamma)^2}\,\Big[-x\,\ln(2x)+\ln\gamma+x+\frac{(1-x)^2+x^2}{2\gamma}-\frac{3}{4}\gamma\Big],
		& x\geq \tfrac{\gamma}{2}\end{cases},
	\end{equation*}
	\item for $x\leq\gamma,\ x\leq 1-\gamma$
	\begin{equation*}\hspace*{-4mm}
	\phi_\gamma(x)=\begin{cases}	
		-\frac{8}{(1-\gamma)^2}\,\Big[(1-x)\,\ln(1-x)+x-\frac{x^2}{2\gamma}\Big],& x\leq \frac{\gamma}{2}\\[1.5mm]
		-\frac{8}{(1-\gamma)^2}\,\Big[(1-x)\,\ln(1-x)-x\,\ln\big(\frac{2x}{\gamma}\big)+x+\frac{x^2}{2\gamma}
		-\frac{\gamma}{4}\Big],& x\geq\frac{\gamma}{2}\\
	\end{cases}.
	\end{equation*}
	\end{enumerate}
	 The corresponding expressions for $x\in[\frac12,1]$ are obtained by replacing $x$ by $1-x$.	
\end{lemma}

\begin{proof}
	As in the proof of Lemma \ref{density-la}, we can take $U$ to be uniform on the set
	$A_\gamma:=\{(y,z)\in\R^2,\ \gamma\leq y\leq 1,\ 0\leq z\leq y-\gamma\}$ and consider the very same linear
	transform
	$T$, see Figure \ref{T_gamma} below.
	\begin{figure}[H]
		\hspace*{1cm}
		\includegraphics[scale=0.9]{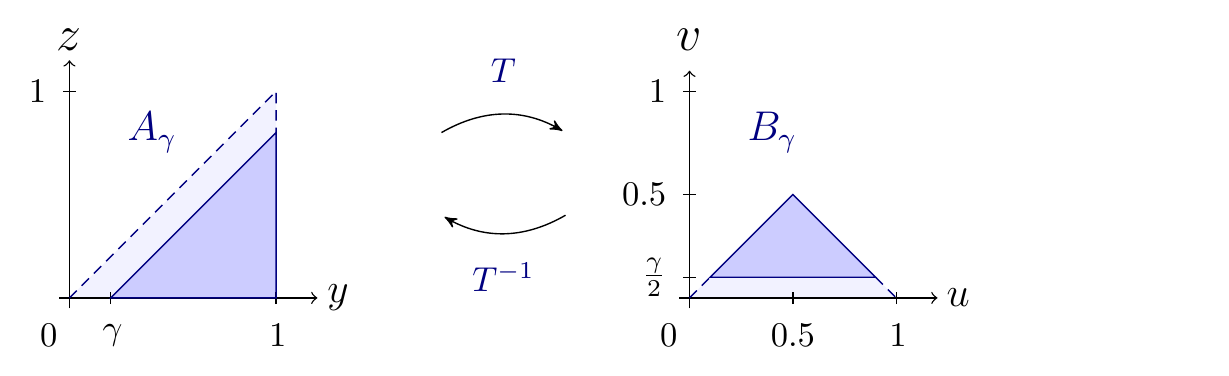}
		\caption{The restricted set $A_\gamma$ forces a minimum amount of open-mindedness.}\label{T_gamma}
	\end{figure}
	
	Then, $T(U)$ is uniform on the set $B_\gamma:=\{(u,v),\ v\in[\tfrac{\gamma}{2},\tfrac12], u\in[v,1-v]\}$
	and the corresponding random density is still
	$$f_{T(U)}(x)=\tfrac{1}{v}\cdot\big(1-\tfrac{1}{v}\cdot\big|x-u\big|\big)^+,\quad x\in[0,1].$$
    Depending on the values of $x\in[0,\tfrac12]$ and $\gamma\in(0,1)$ -- see Figure \ref{areas} for an
    illustration -- quite cumbersome but nevertheless elementary calculations in the same vein as in the
    proof of Lemma \ref{density-la} (which we will leave to the reader to perform) lead to the formulas
    stated above.
    \vspace*{-1.2em}
   	\begin{figure}[H]
		\hspace*{2.9cm}
   		\includegraphics[scale=0.9]{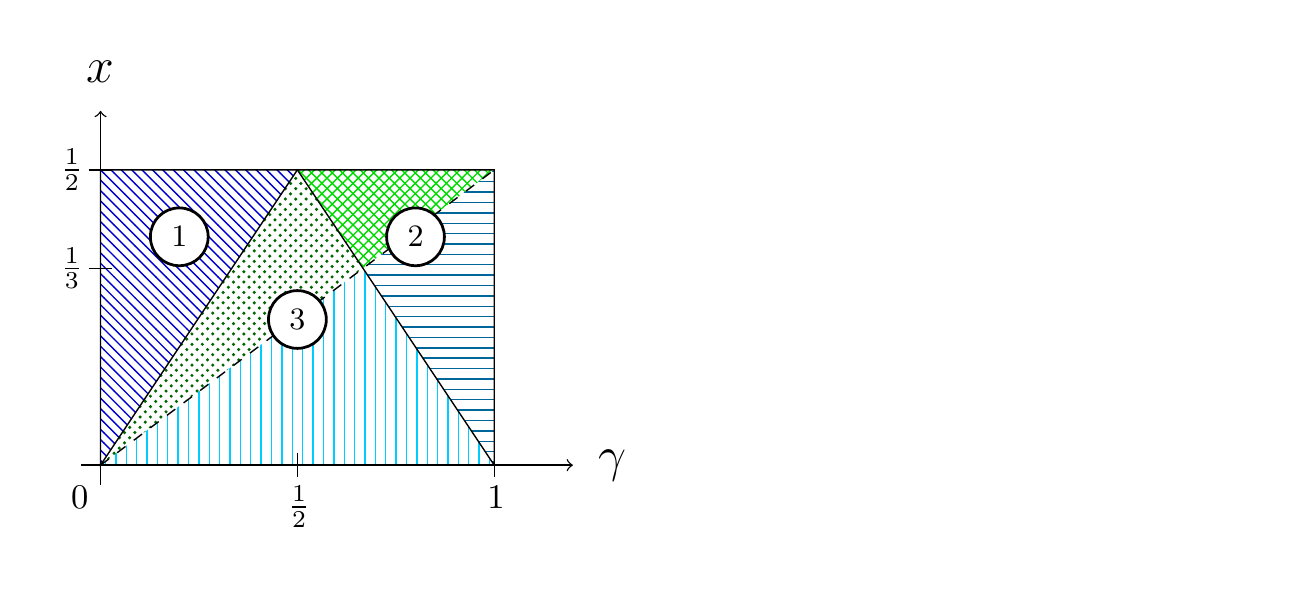}\vspace*{-1.2em}
   		\caption{Different regimes for the form of $\phi_\gamma(x)$.}\label{areas}
   	\end{figure}
    \noindent The last claim follows again by the symmetry in $x$.
\end{proof}

\begin{lemma}\label{monoton}
	Consider $\E\eta_\gamma$ as in the previous lemma.
	Irrespectively of the value of $\gamma\in(0,1)$, the density function $\phi_\gamma(x),\ x\in[0,1],$
	corresponding to the intensity measure, is (strictly) increasing on $[0,\tfrac12]$ and (strictly) decreasing
	on $[\tfrac12,1]$.
\end{lemma}
\begin{proof}
	In principle, one could simply check the expressions for $\phi_\gamma$ given in Lemma \ref{phi_gamma}.
	However, this simple fact can also be seen directly from the construction:
	Let us consider the random density in \eqref{randdens_gamma} to be generated by a vector $T=(U,V)$
	that is chosen uniformly from $B_\gamma$, as described in the proof of Lemma \ref{phi_gamma}. After
	having picked $V\sim\mathrm{unif}([\tfrac{\gamma}{2},\tfrac12])$, we take $U$ to be uniform on $[V, 1-V]$.
	
	Consider $x_1,x_2\in[0,\tfrac12]$, such that $x_1<x_2$, and $V=v$ to be already fixed.
	First note that $f_T(x_1)< f_T(x_2)$ for $U>\frac{x_1+x_2}{2}\in(0,\tfrac12)$.
	If $v\leq\frac{x_1+x_2}{2}$, symmetry around $\frac{x_1+x_2}{2}$ shows that
	$f_{(U,v)}(x_1)$ and $f_{(U,v)}(x_2)$ have the same distribution for $U$ conditioned on
	$[v, x_1+x_2-v]$. In conclusion, we found $f_T(x_1)\prec f_T(x_2)$ and especially
	$$\phi_\gamma(x_1)=\E f_T(x_1)<\E f_T(x_2)=\phi_\gamma(x_2).$$
	Symmetry of $\phi_\gamma$ around $x=\tfrac12$ implies the second part of the claim.
\end{proof}
\vspace*{1em}

Note that $\phi_\gamma$ can not be arbitrarily well approximated by the density of a restricted triangular
distribution. Consequently, for $\epsilon>0$ sufficiently small, there can't be any $\epsilon$-flat vertices
with respect to the initial configuration as all triangular distributions have a positive distance to the
intensity measure $\E\eta_\gamma$ bounded away from 0. For this reason, we have to go the same detour as in
the proof of part (ii) of Thm.\ 2.2 in \cite{Deffuant}.

We need to verify that the density of the intensity measure actually can appear at a later time, more
precisely be arbitrarily well approximated by the opinions that form when agents have interacted. This
happens in fact for all positive values of the model parameter $\theta$:

\begin{lemma}\label{innerpart}
	Consider the Deffuant model on $\Z$ with arbitrary parameter $\theta\in(0,1]$ in which the initial
	opinions are i.i.d.\ absolutely continuous measures given by the random densities described in 
	\eqref{randdens_gamma}. Then, at any time $t>0$ and for any $\epsilon>0$, a (sufficiently long) fixed
	finite section of the infinite path will hold opinions that are less than $\epsilon$ away from
	$\E\eta_\gamma$ in total variation distance (and be bounded by edges on which no Poisson events
	occurred up to time $t$) with positive probability.
\end{lemma}

\begin{proof}	
Fix $\theta\in(0,1]$, $\epsilon>0$ and $t>0$. The idea is to show that a set of agents with suitably assigned
initial opinions can interact in such a way that at time $t$ opinions close to $\E\eta_\gamma$ are formed.

Let us consider an i.i.d.\ sequence $(U_n)_{n\in\N}$ of random variables uniformly distributed on
$A_\gamma=\{(y,z)\in\R^2,\ \gamma\leq y\leq 1,\ 0\leq z\leq y-\gamma\}$, see Figure \ref{T_gamma}.
Then we get the density corresponding to the initial opinion for agent $v\in \N$ by
\begin{equation}\label{initdens}
f_0^{(v)}=f_{U_v}
\end{equation}
where $f_{U_v}$ is taken to be as in \eqref{randdens_gamma} and $f_t^{(v)}$ denotes the random density
corresponding to the opinion of agent $v$ at time $t$.

From Theorem \ref{SLLN} we know that with probability 1
\begin{equation}\label{SLLNdens}
\lim_{n\to\infty} \frac1n \sum_{v=1}^n f_0^{(v)}(x)= \phi_\gamma(x)
\end{equation}
uniformly in $x\in[0,1]$.

It is not hard to check that $\max\{|y_1-y_2|,|z_1-z_2|\}\leq\delta$ entails
\begin{equation}\label{estimate}
	\frac12 \int_0^1\big|f_{(y_1,z_1)}(x)-f_{(y_2,z_2)}(x)\big|\;\mathrm{d}x\leq \frac{2\delta}{\gamma},
\end{equation}
in other words: If the coordinates of two vectors, $(y_1,z_1)$ and $(y_2,z_2)$, shaping restricted
triangular distributions in the sense of \eqref{randdens_gamma} do not differ by more than $\delta$,
the total variation distance between the corresponding measures is at most $\tfrac{2\delta}{\gamma}$.

Fix $m\geq\tfrac{16}{\gamma \theta}$ and subdivide $[0,1]^2$
into $m^2$ squares. The standard SLLN implies that the fraction of $(U_n)_{n\in\N}$ landing
in a square completely contained in $A_\gamma$ a.s.\ tends to $\tfrac{2}{m^2(1-\gamma)^2}>0$ as $n\to\infty$. 
We can therefore choose $N\in\N$ large enough such that, with positive probability, every square that
is a subset of $A_\gamma$ contains at least one of $(U_n)_{n=1}^N$ and
\begin{equation}
\Big\lVert\frac1N \sum_{v=1}^N f_0^{(v)}-\phi_\gamma \Big\rVert_\infty\leq \epsilon
\end{equation}

Note that by symmetry under permutations, there is at least a chance of $\frac{1}{N!}$ that the agents
$\{1,\dots,N\}$
are assigned these values from $A_\gamma$ in such a way that those of neighboring agents do not
differ much in both coordinates; more precisely, matching the values in a serpentine fashion as depicted
in Figure \ref{spread} will keep discrepancies in the $y$-coordinate below $\frac4m$ and in the
z-coordinate below $\frac3m$.

\begin{figure}[h]
	\hspace*{2.5cm}
	\includegraphics[scale=0.8]{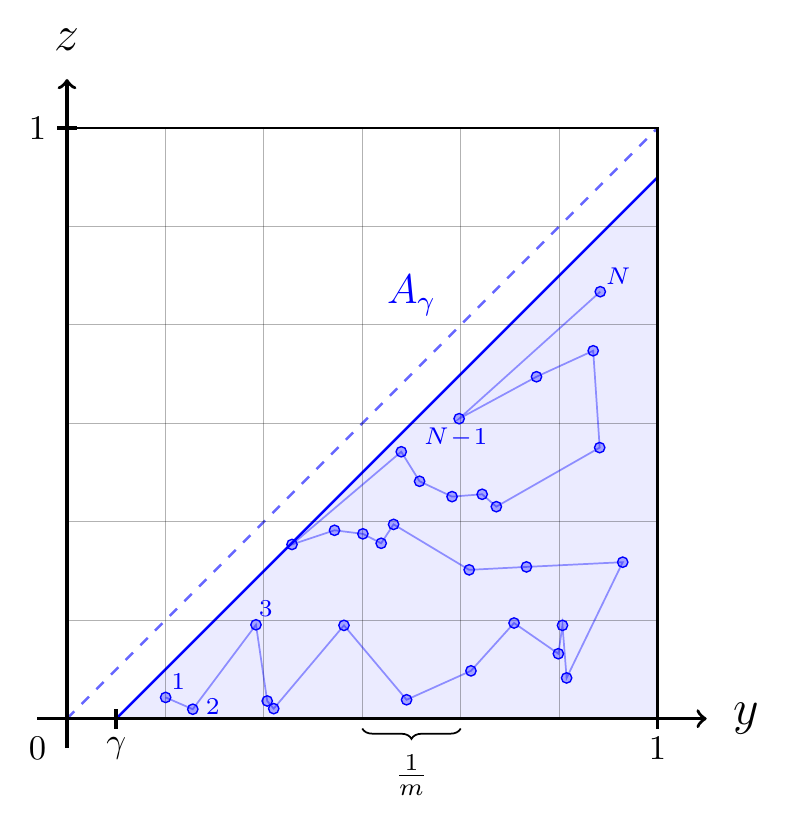}
	\caption{Finding suitable values for the initial opinions that can generate opinions close to
		the intensity measure $\E\eta_\gamma$ on a finite section.}\label{spread}
\end{figure}

Putting things together, we found that for large enough $N$ with non-zero probability the agents
$1$ through $N$ have an initial configuration with a mean at total variation distance at most
$\frac{\epsilon}{2}$ to $\E\eta_\gamma$ and distance at most $\frac{8}{\gamma m}\leq \frac{\theta}{2}$
between neighbors.

Assume that there are no updates on the edges $\langle 0,1\rangle$ and $\langle N,N+1\rangle$ up to
time $t$. It is easy to check (by induction) that in this case, updates on the considered section
in sweeps from left to right, i.e.\ first on $\langle 1,2 \rangle$, then $\langle 2,3 \rangle$ etc.\ 
until $\langle N-1,N\rangle$ repetitively, will keep the total variation distance of neighbors inside
the section always below $\theta$.

The following lemma finally verifies that a sufficiently large number of such sweeps will eventually
bring the considered opinions within total variation distance $\frac{\epsilon}{2}$ of their mean, due
to the fact that the mean is preserved given that there are no updates on neither $\langle 0,1\rangle$
nor $\langle N,N+1\rangle$.

Since the Poisson clocks and the initial configuration are independent, these two events coincide
with positive probability and the claim is verified.
\end{proof}

\vspace*{1em}
\begin{lemma}\label{energyLa}
	If there are infinitely many (performed) updates along an edge, the total variation distance of
	the corresponding neighbors' opinions a.s.\ converges to 0.
\end{lemma}

\begin{proof}
	This statement follows immediately using the energy idea used in the proofs of Thm.\ 2.3 and 5.3
	in \cite{ShareDrink}:
	Consider 
	$$W_t(v):=\int_0^1 \big[f_t^{(v)}(x)\big]^2\,\mathrm{d}x$$
	to be the energy of vertex $v$ at time $t$.
	When an update along the edge $\langle u,v \rangle$ is actually performed, i.e.\ the opinion values
	$\big(\eta_{t-}(u),\eta_{t-}(v)\big)$ get replaced by $\big(\eta_{t}(u),\eta_{t}(v)\big)=
	\big((1-\mu)\cdot\eta_{t-}(u)+\mu\cdot\eta_{t-}(v),(1-\mu)\cdot\eta_{t-}(v)+\mu\cdot\eta_{t-}(u)\big)$
	energy is lost to the amount of
	\begin{align*}
	2\mu(1-\mu)\int_0^1 \big[f_{t-}^{(u)}(x)-f_{t-}^{(v)}(x)\big]^2\,\mathrm{d}x
	&\geq 2\mu(1-\mu)\Big(\int_0^1 \big|f_{t-}^{(u)}(x)-f_{t-}^{(v)}(x)\big|\,\mathrm{d}x\Big)^2\\
	&=8\mu\,(1-\mu)\, \big\lVert\eta_{t-}(u)-\eta_{t-}(v)\big\rVert_{\mathrm{TV}}^{\quad 2}.
	\end{align*}
	As in Lemma 6.2 in \cite{ShareDrink}, define the total energy $W_t^{\mathrm{tot}}(v)$ at
	$v$ to be $W_t(v)$ plus the energy lost on $\langle v,v+1 \rangle$ until time $t$. 
	If we let $X(v)$ denote the random variable consisting of $\eta_0(v)$ and the Poisson process
	associated with the edge $\langle v,v+1\rangle$, $\{X(v),\;v\in\Z\}$ is an i.i.d.\ sequence.
	$W_t^{\mathrm{tot}}(0)$ is a measurable function of $\{X(v),\;v\in\Z\}$ and
	$\{W_t^{\mathrm{tot}}(v),\;v\in\Z\}$ its corresponding shifted equivalents. The well-known
	Pointwise Ergodic Theorem due to Birkhoff-Khinchin thus implies
	$$\E[W_t^{\mathrm{tot}}(0)]=\lim_{\substack{y\to-\infty\\z\to\infty}} \frac{1}{z-y+1}
	\sum_{u=y}^z W_t^{\mathrm{tot}}(u)\quad\text{a.s.}$$
		
	Note that there are a.s.\ infinitely many edges on which no Poisson event has occured up to time $t$
	and that on a section between two such edges, the sum of total energies is preserved until $t$.
	Putting things together, we find
	\begin{align*}\E[W_t^{\mathrm{tot}}(v)]&=\lim_{\substack{y\to-\infty\\z\to\infty}} \frac{1}{z-y+1}
	\sum_{u=y}^z W_t^{\mathrm{tot}}(u)\quad\text{a.s.}\\
	&=\lim_{\substack{y\to-\infty\\z\to\infty}} \frac{1}{z-y+1}	\sum_{u=y}^z W_0^{\mathrm{tot}}(u)\quad\text{a.s.}\\
	&=\E[W_0^{\mathrm{tot}}(v)]=\E[W_0(v)]=-\frac{8}{3\,(1-\gamma)}\,\Big(1+\frac{\ln(\gamma)}{1-\gamma}\Big).
	\end{align*}
	If we assume for contradiction that with positive probability for some $\delta>0$ the total variation distance $\lVert\eta_{t}(v)-\eta_{t}(v+1)\rVert_{\mathrm{TV}}$ lies in $[\delta,\theta]$ for arbitrarily
	large $t$, the conditional Borel-Cantelli lemma (see e.g.\ Cor.\ 6.20 in
	\cite{Foundations}) forces infinitely many performed updates with the total variation distance being
	at least $\delta$. As the total energy is always non-negative, this implies $\lim_{t\to\infty} \E[W_t^{\mathrm{tot}}(v)]=\infty$, a contradiction.
\end{proof}

\vspace*{1em}
Before we can use Lemma \ref{innerpart} to guarantee the existence of flat vertices at time $t>0$,
we need to check that \eqref{SLLNdens} also holds for time $t>0$.

\begin{lemma}\label{outerpart}
	Given the Deffuant model as described in Lemma \ref{innerpart}, for all $t\geq 0$, it
	holds that
	\begin{equation*}
	\lim_{n\to\infty} \frac1n \sum_{v=1}^n f_t^{(v)}= \phi_\gamma
	\end{equation*}
	almost surely with respect to the supremum norm on $[0,1]$.
\end{lemma}

\begin{proof}
	Fix $t>0$. From Theorem \ref{SLLN} we know that 
	$$\lim_{n\to\infty} \frac1n \sum_{v=1}^n f_0^{(v)}= \phi_\gamma\quad\text{a.s.}$$
	with respect to the supremum norm. Using the fact that the densities are uniformly bounded by $\frac{2}{\gamma}$
	we can conclude that this convergence holds even for $t>0$, by the same token as in \cite{Deffuant}:
	
	To the right of site $1$, a.s.\ there is an infinite increasing sequence of nodes $(v_k)_{k\in\N}$, such that
	there was no Poisson event up to time $t$ on the collection of edges $\{\langle v_k,v_k+1\rangle,\;k\in\N\}$.
	We denote the random lengths of the intervals in between by $L_k:=v_{k+1}-v_k,\text{ for } k\in\N$.
	In addition, let $L_0:=v_1-v_0$ be the length of the interval including agent $1$, where
	$\langle v_0,v_0+1\rangle$ is the first edge to the left of site $1$ without Poisson event.
	Independence of the involved Poisson processes entails that $(L_k)_{k\in\N_0}$ is an i.i.d.\
	sequence of random variables having geometric distribution on $\N$ with parameter $\text{e}^{-t}$.
	
	Fix $\delta>0$. Using the Borel-Cantelli lemma we find that the event
	$$E_\delta:=\{L_0=\infty\}\,\cup\, \limsup_{k\to\infty}\Big\{L_k\geq k\cdot\tfrac{\delta\gamma}{2}\Big\}$$
	has probability $0$.
	
	The Deffuant model is mass-preserving in the sense that the sum of opinions of two interacting
	agents is always preserved. Therefore it holds for all $k\in\N$:
	$$\sum_{u=v_0+1}^{v_k}f_0^{(v)}=\sum_{u=v_0+1}^{v_k}f_t^{(u)}.$$
	Furthermore, for some $v\in\{v_k+1,\dots, v_{k+1}\}$, the event
	$$\Big\lVert\frac{1}{v-v_0} \sum_{u=v_0+1}^v f_0^{(u)}-\frac{1}{v-v_0} \sum_{u=v_0+1}^v f_t^{(u)} \Big\rVert_\infty\geq \delta$$
	forces $L_k\geq k\cdot\tfrac{\delta\gamma}{2}$, since
	$v_k\geq k$ and the density $f_s^{(u)}$ is non-negative and uniformly bounded by $\tfrac{2}{\gamma}$
	for all $u\in\Z$ and times $s\geq0$.
	
	In conclusion, given $E_\delta^{\ \mathrm{c}}$, it holds that
	\begin{align*}
	\lim_{n\to\infty}\frac{1}{n}\sum_{v=1}^{n}f_t^{(v)}(x)
	&=\lim_{n\to\infty}\frac{1}{n}\sum_{v=v_0+1}^{n}f_t^{(v)}(x)
	\leq\lim_{n\to\infty}\frac{1}{n}\sum_{v=v_0+1}^{n}f_0^{(v)}(x)+\delta\\
	&=\lim_{n\to\infty}\frac{1}{n}\sum_{v=1}^{n}f_0^{(v)}(x)+\delta
	=\phi_\gamma(x)+\delta
	\end{align*}
	uniformly in $x\in[0,1]$. In the same way we get $\phi_\gamma(x)-\delta$ as a lower bound and
	letting $\delta$ go to $0$ finally verifies
	\begin{equation}
	\lim_{n\to\infty} \frac1n \sum_{v=1}^n f_t^{(v)}= \phi_\gamma\quad\text{a.s.}
	\end{equation}
	w.r.t.\ the supremum norm.
\end{proof}

\vspace*{1em}
\begin{lemma}\label{flatLa}
	Given the Deffuant model as described in Lemma \ref{innerpart} and $\epsilon>0$,
	the following holds for all $t>0$:
	\begin{enumerate}[(i)]
	\item With non-zero probability, there has been no Poisson event on the edge $\langle0,1\rangle$
	until time $t$ and site $1$ is $\epsilon$-flat to the right with respect to the
	configuration $\{\eta_t(v)\}_{v\in\Z}$ and distance measure $\lVert\,.\,\rVert_{\mathrm{TV}}$.
	\item With non-zero probability, site $0$ is two-sidedly $\epsilon$-flat with respect to the
	configuration $\{\eta_t(v)\}_{v\in\Z}$ and distance measure $\lVert\,.\,\rVert_{\mathrm{TV}}$.
	\end{enumerate}
\end{lemma}

\begin{proof}
	In order to verify these claims, we only have to put together the ingredients established in
	Lemmas \ref{innerpart} and \ref{outerpart}. As in the proof of Thm.\ 2.2 in \cite{Deffuant},
	we will do this by using a conditional variant of the coupling technique introduced in \cite{NewSch},
	that became known as {\em local modification} in percolation theory.
	
	Fix $\epsilon>0$. Recall that for absolutely continuous measures $\mu$ and $\nu$ on $[0,1]$, with
	densities $f$ and $g$ respectively, we get $\lVert \mu-\nu\rVert_{\mathrm{TV}}\leq\tfrac12\cdot
	\lVert f-g\rVert_\infty$ and let $B$ denote the following event:
	$$\Big\{\Big\lVert\frac1n \sum_{v=1}^n \eta_t(v)-\E\eta_\gamma\Big\rVert_{\mathrm{TV}}\leq\frac{\epsilon}{3}
	\text{ and }
	\Big\lVert\frac1n \sum_{v=-n}^{-1}\!\!\eta_t(v)-\E\eta_\gamma\Big\rVert_{\mathrm{TV}}\leq\frac{\epsilon}{3},\text{ for all }n\geq N\Big\}.$$
	
	From Lemmas \ref{innerpart} and \ref{outerpart}, we know that $N\in\N$ can be chosen sufficiently
	large such that	$\Prob(B)>1-\mathrm{e}^{-2t}$ and $\Prob(C\cap D)>0$, where $C$ denotes the event
	of no Poisson events on the two edges $\langle 0,1 \rangle$ and $\langle N,N+1 \rangle$
	up to time $t$,
	$$D:=\big\{\lVert\eta_t(v)-\E\eta_\gamma\rVert_{\mathrm{TV}}\leq \tfrac{\epsilon}{3}\text{ for all }
	1\leq v\leq N\big\}.$$
	Additionally, since $\Prob(C)=\mathrm{e}^{-2t}$, we must have $\Prob(B\cap C)>0$.
	
	Now let	$\boldsymbol{\eta}_t:=\{\eta_t(v)\}_{v\in\Z}$ and $\boldsymbol{\eta}_t':=\{\eta_t'(v)\}_{v\in\Z}$
	be the configurations at time $t$ originated from two independent copies of the considered model. There
	is a strictly positive probability that $B\cap C$ happens for $\boldsymbol{\eta}_t$ as well as that
	$C\cap D$ happens for $\boldsymbol{\eta}_t'$.
	Given $C$, the hybrid process $\tilde{\boldsymbol{\eta}}_t$ defined by
	\begin{equation*}
	\tilde{\eta}_t(v)=\begin{cases}\eta_t(v)& \text{if }v\notin\{1,\dots,N\}\\
							       \eta_t'(v)& \text{if } v\in\{1,\dots,N\} \end{cases}
	\end{equation*}
	is a perfectly fine copy of the model as well, showing that the event $B\cap D$ has non-zero
	probability. It is an easy exercise to check that $B\cap D$ actually implies the $\epsilon$-flatness
	to the right of site $1$.
	
	In fact, the same argument applies to the second setting. Here, however, we choose $C$ to be the event
	that there were no Poisson events on $\langle -N-1,-N \rangle$ and $\langle N,N+1 \rangle$ as well as
	$D:=\big\{\lVert\eta_t(v)-\E\eta_\gamma\rVert_{\mathrm{TV}}\leq \tfrac{\epsilon}{3}\text{ for all }
	-N\leq v\leq N\big\}$. Then the two-sidedly $\epsilon$-flatness of site $0$ follows from $B\cap D$.
\end{proof}

\vspace*{1em}
Let us now use Lemmas \ref{energyLa} and \ref{flatLa} to prove the main statement about the model featuring restricted
random triangular distributions.

\begin{theorem}\label{rrtd}
Consider the Deffuant model on $\Z$, in which the total variation distance is used to measure the
difference between two opinions and in which the initial opinions are given by independently
assigned random restricted triangular distributions with fixed $\gamma\in(0,1)$ as described in
\eqref{randdens_gamma}. Then there is a sharp phase transition in the following sense: for $\theta\in[0,\theta_{\mathrm{c}})$, the system almost surely approaches
disagreement in the long run; for $\theta\in (\theta_{\mathrm{c}},1]$, it almost surely approaches consensus.
The threshold $\theta_{\mathrm{c}}$ is given by
\begin{equation}
\begin{split}&\theta_{\mathrm{c}}=\frac12\,\int_0^1 \Big|f_{(\gamma,0)}(x)-\phi_\gamma(x)\Big|\;\mathrm{d}x\\
   &\phantom{\theta_{\mathrm{c}}}=\frac12\,\int_0^1 \Big|\tfrac{2}{\gamma}\cdot\big(1-\tfrac{2}{\gamma}\cdot
                    \big|x-\tfrac{\gamma}{2}\big|\big)^+-\phi_\gamma(x)\Big|\;\mathrm{d}x.\end{split}
\end{equation}
\end{theorem}

\begin{proof}
As mentioned in Section \ref{background}, we will closely follow the ideas in \cite{ShareDrink} to
establish this result, just now the opinions are given by (random) absolutely continuous measures,
or rather their density functions.
In fact, most of the work has already been done by showing Lemma \ref{flatLa}. 
Let us define
 $$\theta_{\mathrm{c}}:=\frac12\,\int_0^1 \Big|f_{(\gamma,0)}(x)-\phi_\gamma(x)\Big|\;\mathrm{d}x$$
and $\epsilon:=\tfrac17\cdot|\theta_{\mathrm{c}}-\theta|$. In the sequel, we will consider the two regimes:
the subcritical one ($\theta<\theta_{\mathrm{c}}$) and the supercritical one ($\theta>\theta_{\mathrm{c}}$).
\vspace*{1em}

In the subcritical regime, fix $t>0$ and let $B$ denote the event that there are no Poisson events
neither on $\langle -1,0 \rangle$ nor on $\langle 0,1 \rangle$ during $[0,t]$ and site $-1$ is
$\epsilon$-flat to the left, site $1$ is $\epsilon$-flat to the right with respect to the configuration $\{\eta_t(v)\}_{v\in\Z}$.
By Lemma \ref{flatLa} part (i), the obvious symmetry and conditional independence, we know that $B$ occurs
with positive probability.
Let the initial opinion of agent $0$ be given by $f_{U_0}$ in the sense of \eqref{initdens} and $C$ be
the event that the first coordinate of $U_0$ is less than $\gamma+\tfrac{\epsilon\gamma}{2}$, which
by the shape of $A_\gamma$ (see Figure \ref{T_gamma}) and \eqref{estimate} entails 
$$\big\lVert\eta_0(0)-\E\eta_\gamma\big\rVert_{\mathrm{TV}}\geq \theta_{\mathrm{c}}-\epsilon.$$
Given that there are no Poisson events on edges incident to site $0$, we can (again by local modification)
conclude that $B\cap C$ has positive probability. From Lemma \ref{flat}, we know that the total
variation distance between $\eta_s(1)$ and the intensity measure $\E\eta_\gamma$ will not exceed
$\epsilon$ for $s\geq t$ due to its one-sided $\epsilon$-flatness if there is no interaction with
site $0$ (same for $\eta_s(-1)$). However, given $B\cap C$ the opinions $\eta_t(1)$ and $\eta_t(0)=\eta_0(0)$
are at distance larger than $\theta_{\mathrm{c}}-2\epsilon>\theta$ and hence they never will be close
enough to interact, since the same holds for $\eta_t(-1)$, which leaves the opinion at site $0$ unchanged
for all time. In other words, with non-zero probability the edge $\langle 0,1\rangle$ will be finally blocked
(same for $\langle -1,0\rangle$).

To conclude the claimed almost sure behavior, we apply the ergodicity argument used in the proof of Lemma
\ref{energyLa} once again: Whether the configuration approaches disagreement or not can be checked given
the initial configuration plus all Poisson processes associated to the edges. The sequence $\{X(v),\;v\in\Z\}$
(as defined in the proof of Lemma \ref{energyLa}) is i.i.d., hence ergodic with respect to shifts. Thus the
translation-invariant event ``disagreement'' necessarily has to be trivial, i.e.\ must have probability
either $0$ or $1$. Since we already showed that its probability is non-zero, the event has to be an almost
sure one in the subcritical regime.

In the supercritical case, we know that at time $t>0$ any fixed site is two-sidedly $\epsilon$-flat with
positive probability (part (ii) of Lemma \ref{flatLa}). Lemma \ref{monoton} implies that the largest
total variation distance of a restricted triangular distribution as defined in \eqref{randdens_gamma}
to the intensity measure $\E\eta_\gamma$ is given by what we defined as $\theta_{\mathrm{c}}$. Since all
opinions at a later time are convex combinations of the initial ones, this distance can not be exceeded.
If site $0$ is two-sidedly $\epsilon$-flat with respect to the configuration $\{\eta_t(v)\}_{v\in\Z}$, 
the opinion $\eta_s(0)$ will be at total variation distance at most $6\epsilon$ to $\E\eta_\gamma$, for
all $s\geq t$ (part (ii) of Lemma \ref{flat}). This implies that its neighboring opinions differ by not
more than $\theta_{\mathrm{c}}+6\epsilon<\theta$, hence Lemma \ref{energyLa} forces the differences
to converge to $0$. By induction, this is actually true for any pair of neighbors. Again, ergodicity
ensures that this consensus behavior occurs with probability 1.

It further implies the almost sure existence of two-sidedly $\epsilon$-flat vertices for any strictly
positive value of $\epsilon$. For this reason, the measure, the opinions converge to, must be the intensity
measure, which concludes the proof.
\end{proof}
\vspace*{1em}

\begin{example}
Let us consider the Deffuant model on $\Z$ with opinions being absolutely continuous probability
distributions and the initial ones given by random restricted triangular distributions with parameter
$\gamma=\tfrac13$. From Lemma \ref{phi_gamma} we know that the corresponding intensity measure
has the somewhat cumbersome density

\begin{equation*}
\phi_{\tfrac13}(x)=\begin{cases}
-18\,\big[(1-x)\,\ln(1-x)-\tfrac32\,x^2+x\big],& \,0\leq x\leq \tfrac16\\
-18\,\big[(1-x)\,\ln(1-x)-x\,\ln(6x)+x+\tfrac32\,x^2-\tfrac{1}{12}\big],& \tfrac16\leq x\leq \tfrac13 \\
-18\,\big[(1-x)\,\ln(1-x)+x\,(1-\ln(2))+\tfrac{1}{12}\big],& \tfrac13\leq x\leq \tfrac12 \\
-18\,\big[x\,\ln(x)+(1-x)\,(1-\ln(2))+\tfrac{1}{12}\big],& \tfrac12\leq x\leq \tfrac23 \\
-18\,\big[x\,\ln(x)-(1-x)\,\ln(6-6x)+\tfrac32\,x^2-4x+\tfrac{29}{12}\big],
& \tfrac23\leq x\leq \tfrac56 \\
-18\,\big[x\,\ln(x)-\tfrac32\,x^2+2x-\tfrac{1}{2}\big],& \tfrac56\leq x\leq 1.
\end{cases}
\end{equation*}
\vspace*{-1.5em}
   	\begin{figure}[H]
   		\hspace*{3.2cm}
   		\includegraphics[scale=0.8]{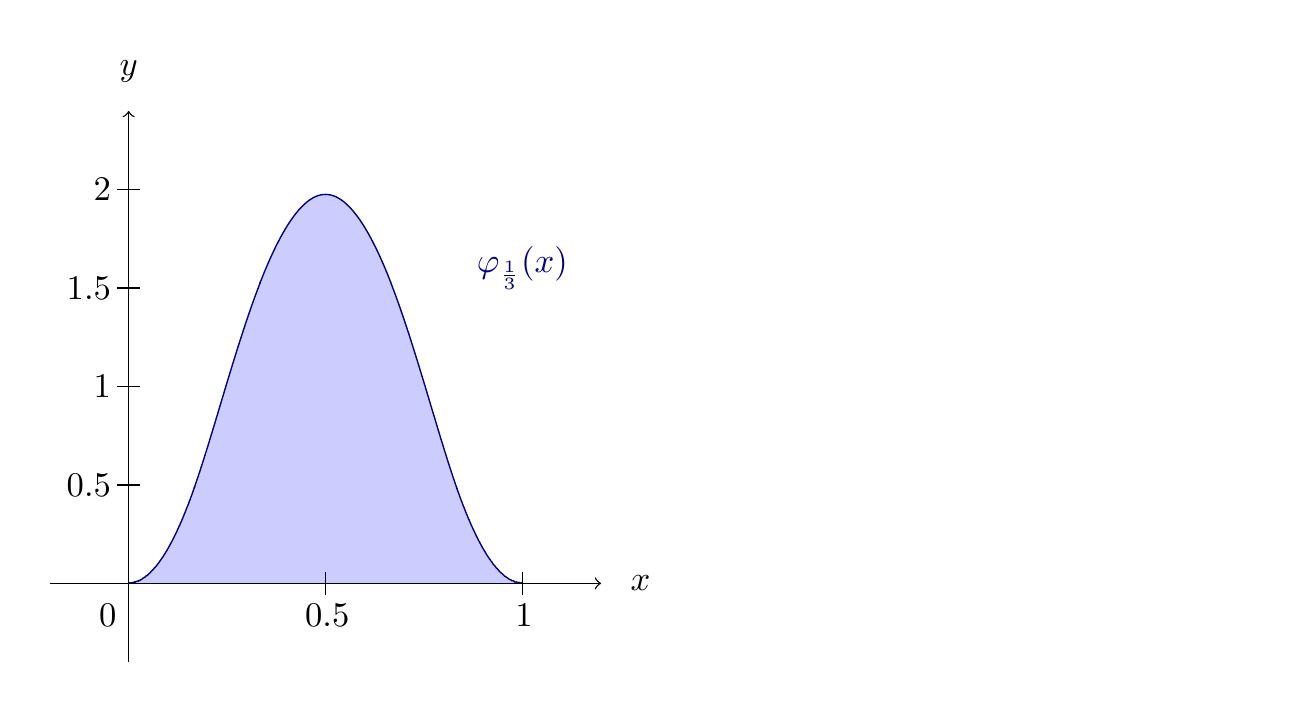}\vspace*{-0.5em}
   		\caption{The intensity measure is concentrated more towards the center if the random triangular
   			distributions are restricted to a minimum width.}\label{exampl}
   	\end{figure}

If the total variation distance is used to measure the disparity of two opinions, we
can conclude from Theorem \ref{rrtd} that the threshold $\theta_{\mathrm{c}}$ for this model takes
on the value 
$$\theta_{\mathrm{c}}=\frac12\int_0^1 \Big|f_{\big(0,\tfrac13\big)}(x)-\phi_{\tfrac13}(x)\Big|\;\mathrm{d}x
					 \approx 0.83172.$$

\end{example}

\section{Alternative choices (concerning initial configuration and distance measure) and inhomogeneous
	open-mindedness}

No doubt that the extension of the Deffuant model on $\Z$ to measure-valued opinions leaves a wide range
of possible laws for the initial configuration to be examined. We saw trivial behavior for triangular
distributions and a non-trivial phase transition for restricted triangular distributions.

If we want to stick to absolutely continuous measures on a compact support $S$ and
$\lVert\,.\,\rVert_{\mathrm{TV}}$ as distance measure, the line of argument from Section \ref{sec:rrtd}
will in principle carry over on condition that
$$\E\int_S [f_\omega(x)]^2\,\mathrm{d}x<\infty,$$
where $f_\omega,\ \omega\in\Omega,$ denotes the random initial density shaped by a probability
space $(\Omega,\Prob)$, and that the total variation distance of an initial opinion to the intensity
measure is a.s.\ bounded away from $1$. The first condition is needed to prove Lemma \ref{energyLa},
the latter will in fact give the threshold $\theta_{\mathrm{c}}$, as essential supremum of the
total variation distance of initial opinions to the intensity measure.
There is however one issue, that must not be overlooked: In order to establish Lemma \ref{innerpart},
we needed that there are no major gaps in the support of the initial opinions -- just as in the case
of finite-dimensional opinions. To examine this problem more closely, elaborate geometric considerations
as in \cite{multidim} seem to be necessary and we will thus leave this for future studies.

If one wants to include point processes as opinions, in many situations the total variation distance
will not work as a meaningful measure for the discrepancy of two opinions, at least if the support
of two such processes is disjoint with positive probability. Using the Lévy-distance instead could
however lead to interesting models in such a setting. In fact, even in the case of triangular
distributions, it seems to be unrealistic that two determined agents are at maximal distance, whether
the intervals, on which their opinions are concentrated, are in close proximity or at different ends
of the spectrum. From this point of view, albeit more difficult to handle, the Lévy-distance appears
to be a more suitable choice.

Finally, it might be interesting to point out the new feature of the model (compared to finite-dimensional
opinions) that was mentioned in the beginning of Section \ref{sec:td} once more and put it into a broader
context: In the Deffuant model with triangular distributions, besides the common tolerance parameter
$\theta$ and willingness to compromise $\mu$, we have a diversified scale of open-mindedness of the agents
shaped by the random support of their initial opinion.

There have been attempts, e.g.\ by Deffuant et al.\ in \cite{heterogeneous}, to simulate the long-term
behavior of a variant of the model in which the agents have different $\theta$-values. On the analytical
side, this is quite a crucial change since it brings along situations in which the opinion of only one
of two interacting neighbors is updated. Then the sum of opinions is no longer preserved, which renders
void many of our central arguments.
Another advance in the same direction is the so-called {\em relative agreement model}, introduced in
\cite{RA}. There, the bounded confidence rule is dropped and replaced by a continuous counterpart:
Agents feature both a real-valued opinion and a separate value corresponding to their
individual uncertainty, which taken together shape a dispersed opinion in the form of a symmetric interval
of length two times the uncertainty around the opinion value. If an agent gets influenced by another, the
impact depends on the overlap of the two opinion intervals relative to the length of the interval
corresponding to the influenced agent. Again, the asymmetric way of updating opinion values makes the
relative agreement model, although based on the same principles and ideas, qualitatively quite different.

On the modelling side, the way inhomogeneous open-mindedness or uncertainty is incorporated in our extension of the
model does not only avoid this issue, but also lead to the realistic property that agents themselves
become more open-minded by interacting with open-minded neighbors.

\subsection*{Acknowledgements}

I am very grateful to my supervisor Olle Häggström for his valuable comments to an earlier draft and
his constant support. Furthermore, I would like to thank the anonymous referee reviewing the paper
dealing with higher-dimensional opinion spaces \cite{multidim}, who suggested to carry on the arguments
beyond finite dimensionality.


\vspace{0.5cm}
\makebox[0.8\textwidth][l]{
	\begin{minipage}[t]{\textwidth}
	{\sc \small Timo Hirscher\\
   Department of Mathematical Sciences,\\
   Chalmers University of Technology,\\
   412 96 Gothenburg, Sweden.}\\
   hirscher@chalmers.se
	\end{minipage}}

\end{document}